\documentclass{amsart}
\usepackage{amssymb,amsmath, amsfonts, amsrefs, tikz, epsfig, float}
\usepackage{amsthm}
\usepackage{mathrsfs}
\usepackage{enumitem, indentfirst}
\usepackage[fleqn,tbtags]{mathtools}
\usepackage{color}
 \newtheorem{theorem}{Theorem}[section]
 \newtheorem{corollary}[theorem]{Corollary}
 
 \newtheorem{proposition}[theorem]{Proposition}

 \theoremstyle{definition}

 \theoremstyle{remark}
 
  \numberwithin{equation}{section}

\renewcommand{\phi}{\varphi}
\renewcommand{\theta}{\vartheta}

\DeclareMathOperator{\tform}{\mathfrak{t}}

\DeclareMathOperator{\wform}{\mathfrak{w}}

\DeclareMathOperator{\mul}{mul}
\DeclarePairedDelimiterX\sipt[2]{(}{)_{\tform}}{#1\,\delimsize\vert\,#2}
\DeclarePairedDelimiterX\sipv[2]{(}{)_{v}}{#1\,\delimsize\vert\,#2}
\DeclarePairedDelimiterX\sipw[2]{(}{)_{w}}{#1\,\delimsize\vert\,#2}

\newcommand{\pair}[2]{\begin{bmatrix}#1 \\  #2 \end{bmatrix}}

\newcommand{\dupR}{\mathbb{R}}

\newcommand{\dupK}{\mathbb{K}}

\newcommand{\dom}{\operatorname{dom}}

\newcommand{\ran}{\operatorname{ran}}

\newcommand{\hil}{\mathcal{H}}

\newcommand{\kil}{\mathcal{K}}

\DeclarePairedDelimiterX\abs[1]{\lvert}{\rvert}{#1}
\DeclarePairedDelimiterX\sip[2]{(}{)}{#1\,\delimsize\vert\,#2}
\DeclarePairedDelimiterX\siptilde[2]{(}{)_{\!_{\widetilde{A}}}}{#1\,\delimsize\vert\,#2}
\DeclarePairedDelimiterX\sipf[2]{(}{)_{f}}{#1\,\delimsize\vert\,#2}
\DeclarePairedDelimiterX\sipg[2]{(}{)_{g}}{#1\,\delimsize\vert\,#2}
\DeclarePairedDelimiterX\siptw[2]{(}{)_{\tform+\wform}}{#1\,\delimsize\vert\,#2}
\DeclarePairedDelimiterX\set[2]{\{}{\}}{#1\,:\,#2}
\DeclarePairedDelimiterX\dual[2]{\langle}{\rangle}{#1,#2}
\DeclarePairedDelimiterX\sipa[2]{(}{)_{\!_A}}{#1\,\delimsize\vert\,#2}
\DeclarePairedDelimiterX\sipc[2]{(}{)_{\!_C}}{#1\,\delimsize\vert\,#2}
\DeclarePairedDelimiterX\sipab[2]{(}{)_{\!_{A+B}}}{#1\,\delimsize\vert\,#2}
\DeclarePairedDelimiterX\sipb[2]{(}{)_{\!_B}}{#1\,\delimsize\vert\,#2}

\newcommand{\opmatrix}[4]{\begin{bmatrix} #1 &  #2 \\   #3& #4\end{bmatrix}}

\allowdisplaybreaks
\begin{document}
\title[Range-kernel characterization of operators]{Range-kernel characterizations of operators which are adjoint of each other}

\author[Zs. Tarcsay]{Zsigmond Tarcsay}

\thanks{The corresponding author Zs. Tarcsay was supported by DAAD-TEMPUS Cooperation Project ``Harmonic Analysis and Extremal Problems'' (grant no. 308015). Project no. ED 18-1-2019-0030 (Application-specific highly reliable IT
solutions)
has been implemented with the support provided from the National Research,
Development and Innovation Fund of Hungary, financed under the Thematic
Excellence
Programme funding scheme.}
\address{%
Department of Applied Analysis and Computational Mathematics\\ E\"otv\"os L. University\\ P\'azm\'any P\'eter s\'et\'any 1/c.\\ Budapest H-1117\\ Hungary}

\email{tarcsay@cs.elte.hu}
\author[Z. Sebesty\'en]{Zolt\'an Sebesty\'en}

\address{%
Department of Applied Analysis and Computational Mathematics\\ E\"otv\"os L. University\\ P\'azm\'any P\'eter s\'et\'any 1/c.\\ Budapest H-1117\\ Hungary}

\email{sebesty@cs.elte.hu}



\begin{abstract}
We provide necessary and sufficient conditions for a pair $S,T$ of Hilbert space operators in order that they satisfy  $S^*=T$ and $T^*=S$. As a main result we establish an improvement of von Neumann's classical theorem on the positive self-adjointness of $S^*S$ for  two variables. We also give some new characterizations of self-adjointness and skew-adjointness of operators, not requiring their symmetry or skew-symmetry, respectively. 
\end{abstract}
\dedicatory{Dedicated to Professor Franciszek Hugon Szafraniec on the occasion of his 80th birthday}

\maketitle

\section{Introduction}
The adjoint of an unbounded linear operator was first introduced by John von Neumann in \cite{vonNeumann1930} as a profound ingredient for developing a rigorous mathematical framework for quantum mechanics. By definition, the adjoint of a densely defined linear transformation $S$, acting between two Hilbert spaces, is an operator $T$ with the largest possible domain such that % that satisfies the adjoint relation
\begin{equation}\label{E:1.1}
    \sip{Sx}{y}=\sip{x}{Ty}
\end{equation}
holds for every $x$ from the domain of $S$. The adjoint  operator, denoted by $S^*$, is therefore ``maximal''  in the sense that it extends every operator $T$ that has property \eqref{E:1.1}. On the other hand, every restriction $T$ of $S^*$ fulfills that adjoint relation. Thus, in order to decide whether an operator $T$ is identical with the adjoint of $S$ it seems  reasonable  to restrict ourselves to investigating those operators $T$ that have property \eqref{E:1.1}. This issue was  explored in detail in \cite{SZTZS2019} by means of the operator matrix 
\begin{equation*}
    \opmatrix{I}{-T}{S}{I},
\end{equation*}
cf. also \cites{Popovici,Characterization,Charess, SZTZS2015}. 

In the present paper we continue to examine the conditions under which an operator $T$ is  equal to the adjoint $S^*$ of $S$. Nevertheless, as opposed to the situation treated in the cited papers, we do not assume that $S$ and $T$ are adjoint to each other in the sense of \eqref{E:1.1}. Observe that condition \eqref{E:1.1} is equivalent to identity 
\begin{equation}\label{E:1.2}
    S^*\cap T=T.
\end{equation}
So, still under condition \eqref{E:1.1}, $T$ is equal to the adjoint  of $S$ if and only if $S^*\cap T=S^*$. 
In the present paper we are going to guarantee equality $S^*=T$ by imposing new conditions, weaker than \eqref{E:1.1},   by means of the kernel and range spaces.
 Roughly speaking,  we only require that the intersection of the graphs of $S^*$ and $T$ be, in a sense,  ``large enough''. We also establish a criterion in terms of the norm of the resolvent of the operator matrix 
 \begin{equation*}
     M_{S,T}=\opmatrix{0}{-T}{S}{0}.
 \end{equation*}
As an application we gain some characterizations of self-adjoint, skew-adjoint and unitary operators, thereby generalizing some analogous results by T. Nieminen \cite{Nieminen} (cf. also \cite{Sandovici}).

%if we are about to determine $S^*$, it seems reasonable  to restrict ourselves to investigating those operators $T$ that have property \eqref{E:1.1}.  

\section{Preliminaries}
Throughout this paper $\hil$ and $\kil$ will denote  real or complex Hilbert spaces. By an operator $S$ between $\hil$ and $\kil$ we mean a linear map $S:\hil\to\kil$ whose domain $\dom S$ is a linear subspace of $\hil$. We stress that, unless otherwise indicated, linear operators are not assumed to be densely defined. However, the adjoint of such an operator can only be interpreted as a ``multivalued operator'', that is, a linear relation. Therefore we are going to collect here some basic notions and facts on linear relations.  

A linear relation between two Hilbert spaces $\hil$ and $\kil$ is nothing but a linear subspace $S$ of the Cartesian product $\hil\times\kil$, respectively, a closed linear relation is just a closed subspace of $\hil\times \kil$. To a linear relation $S$ we associate the following subspaces
\begin{align*}
\dom S&=\set{h\in\hil}{(h,k)\in S} &\ran S&=\set{k\in\kil}{(h,k)\in S}\\
\ker S&=\set{h\in\hil}{(h,0)\in S}  &\mul S&=\set{k\in\kil}{(0,k)\in S},
\end{align*}
which are referred to as the domain, range, kernel and multivalued part of $S$, respectively. Every linear operator when identified with its graph is a linear relation with trivial multivalued part. Conversely, a linear relation whose multivalued part consists only of the vector $0$ is (the graph of) an operator.

A notable advantage of linear relations, compared to operators, lies in the fact that one might define the adjoint without any further assumption on the domain. Namely, the adjoint of a linear relation $S$ will be again a linear relation $S^*$ between $\kil$ and $\hil$,  given by 
\begin{align*}
S^*\coloneqq V( S)^{\perp}.
\end{align*}
Here, $V:\hil\times \kil\to\kil\times\hil$ stands for the `flip' operator $V(h,k)\coloneqq (k,-h)$. It is seen immediately that  $S^*$ is automatically a closed linear relation and satisfies the useful identity 
\begin{align*}
\overline{S}=S^{**}(=:(S^*)^*).
\end{align*}
On the other hand, a closed linear relation $S$ entails  the following orthogonal decomposition of the product Hilbert space $\kil\times \hil$:
\begin{equation*}
    S^*\oplus V(S)=\kil\times \hil.
\end{equation*}
Note that another equivalent definition of $S^*$ is obtained in terms of the inner product as follows:  
\begin{align*}
S^*=\set{(k',h')\in\kil\times\hil}{\sip{k}{k'}=\sip{h}{h'}~\mbox{for all $(h,k)\in S$}}.
\end{align*}
 In other words, $(k',h')\in S^*$ holds if and only if 
\begin{equation*}
    \sip{k}{k'}=\sip{h}{h'}\qquad \forall (h,k)\in S.
\end{equation*}
In particular, if $S$ is a densely defined operator then the relation $S^*$ coincides with the usual adjoint operator of $S$. Recall also the dual identities
\begin{align*}
    \ker S^*=(\ran S)^{\perp},\qquad \mul S^*=(\dom S)^{\perp},
\end{align*}
where the second equality tells us that the adjoint of a densely defined linear relation is always a (single valued) operator. For further information on linear relation we refer the reader to \cites{Arens,BehrndtHassideSnoo,Hassi2009a,Schmudgen}.
%%%%%%%%%%%%%%%%%%%%%%%%%%%%%%%%%%%%%%%%%%%%%%%%%%%%%%%%%%%%%%%%%%%%%%%%%%%%%%%%%%%%%%%%%%%%%%%%%%%%%%%%%%
\section{Operators which are adjoint of each other}

R. Arens \cite{Arens} characterized the equality $S=T$ of two linear relations in terms of their kernel and range (see Corollary \ref{C:Arens2}). Below we provide a similar characterization of $S\subset T$. Observe that the intersection $S\cap T$ of the linear relations $S$ and $T$ is again a linear relation, but this is not true for their union $S\cup T$ as it is not a linear subspace in general. The linear span of $S\cup T$ will be denoted by $S\vee T$, which in turn is a linear relation. 
\begin{proposition}\label{P:Arens}
Let $S$ and $T$ be linear relations between two vector spaces. Then the following three statements are equivalent:
\begin{enumerate}[label=\textup{(\roman*)}]
\item $S\subset T$,
\item $\ker S\subset \ker T$ and $\ran S\subset \ran (S\cap T)$,
\item $\ran S\subset \ran T$ and $\ker (S\vee T)\subset \ker T$.
\end{enumerate}
\end{proposition}
\begin{proof}
It is clear that (i) implies both (ii) and (iii).  Suppose now (ii) and let $(h,k)\in S$ then there exists $u$ with $(u,k)\in T\cap S$. Consequently, $(h-u,0)\in S$, i.e., $h-u\in \ker S\subset \ker T$. Hence 
\begin{align*}
(h,k)=(h-u,0)+(u,k)\in T+T\subset  T,
\end{align*}
which yields $S\subset T$, so (ii) implies (i). Finally, assume (iii) and take $(h,k)\in S$. Then $(u,k)\in T$ for some $u$ and hence $(h-u,0)\in S\vee T$, i.e., $h-u\in \ker T$. Consequently, 
\begin{equation*}
    (h,k)=(h,k)=(h-u,0)+(u,k)\in  T,
\end{equation*}
which yields $S\subset T$. 
\end{proof}
\begin{corollary}\label{C:Arens1}
Let $S$ and $T$ be two linear relations between the vector spaces. The following three statements are equivalent:
\begin{enumerate}[label=\textup{(\roman*)}]
    \item $S=T$,
    \item $\ker S=\ker T$ and $\ran S+\ran T\subseteq \ran (S\cap T)$,
    \item $\ran S=\ran T$ and $\ker (S\vee T)\subseteq \ker (S\cap T)$.
\end{enumerate}
\end{corollary}
%Sometimes it is useful to rephrase the above corollary to the case when $S\subset T$:
\begin{corollary}\label{C:Arens2}
Let $S$ and $T$ be linear relations between two vector spaces such that $S\subset T$. Then the following assertions are equivalent:
\begin{enumerate}[label=\textup{(\roman*)}]
    \item $S=T$,
    \item $\ker S=\ker T$ and $\ran S=\ran T$.
\end{enumerate}
\end{corollary}

%%%%%%%%%%%%%%%%%%%%%%%%%%%%%%%%%%%%%%%%%%%%%%%%%%%%%%%%%%%%%%%%%%

In \cite{Stone}*{Theorem 2.9}  M. H. Stone established a simple yet effective sufficient condition for an operator to be self-adjoint: a densely defined symmetric operator $S$ is necessarily self-adjoint provided it is surjective. In  that case, it is invertible with bounded and self-adjoint inverse due to the Hellinger--Toeplitz theorem. Here, density of the domain can be dropped from the hypotheses: a surjective symmetric operator is automatically densely defined (see also \cite{SZTZS2019}*{Corollary 6.7} and \cite{Squareroot}*{Lemma 2.1}). 

Below we establish a generalization of Stone's result for a pair of operators.   
\begin{proposition}\label{P:genStone}
Let $\hil,\kil$ be real or complex Hilbert spaces  and let $S:\hil\to\kil$ and $T:\kil\to\hil$ be (not necessarily densely defined or closed) linear operators such that 
\begin{equation*}
    \ran (S\cap T^*)=\kil\qquad\mbox{and}\qquad \ran (T\cap S^*)=\hil.
\end{equation*}
Then $S$ and $T$ are both densely defined operators such that  $S^*=T$ and $T^*=S$.
\end{proposition}
\begin{proof}
For brevity, introduce the following notations
\begin{equation*}
    S_0\coloneqq S\cap T^*,\qquad T_0\coloneqq T\cap S^*.
\end{equation*}
Observe that $S_0$ and $T_0$ are adjoint to each other in the sense that
\begin{equation*}
    \sip{S_0x}{y}=\sip{x}{T_0y},\qquad x\in\dom S_0,y\in\dom T_0.
\end{equation*} 
We claim that $S_0$ and $T_0$ are densely defined: let $z\in(\dom S_0)^{\perp}$, then by surjectivity, $z=T_0v$ for some $v\in \dom T_0$. Hence 
 
\begin{equation*}
    0=\sip xz=\sip{x}{T_0v}=\sip{S_0x}{v},\qquad x\in\dom S_0,
\end{equation*}
which implies $v=0$ and also $z=0$. The same argument shows that $T_0$ is densely defined too. %It is now clear that $S_0^*$ and $T_0^*$ are both (single-valued) operators such that $ T^{}_0\subset  S_0^*$ and $ S^{}_0\subset T_0^*$, hence 
We see now that $S$ and $T^*$ are densely defined operators such that 
\begin{equation*}
    \ker S\subseteq (\ran S^*)^\perp=\{0\}, \qquad \ker T^*= (\ran T)^\perp=\{0\},
\end{equation*}
and $\ran (S\cap T^*)=\kil$. Corollary \ref{C:Arens1} applied to $S$ and $T^*$ implies that $S=T^*$. The same argument yields equality $S^*=T$.
\end{proof}

\begin{corollary}
Let $S:\hil\to\kil$ and $T:\kil\to\hil$ be (not necessarily densely defined) surjective operators  such that 
\begin{equation*}
    \sip{Sx}{y}=\sip{x}{Ty},\qquad x\in\dom S,y\in\dom T.
\end{equation*}
Then $S$ and $T$ are both densely defined operators such that $S^*=T$ and $T^*=S$.
\end{corollary}
From Proposition \ref{P:genStone} we  gain a sufficient condition of self-adjointness without the assumptions of being symmetric or densely defined:
%Using the previous proposition  we gain a sufficient condition  for the self-adjoint\-ness of a (not necessarily symmetric) operator: 
\begin{corollary}
Let $\hil$ be a Hilbert space and let $S:\hil\to\hil$ be a linear operator such that $\ran (S\cap S^*)=\hil$. Then $S$ is densely defined and self-adjoint.
\end{corollary}
\begin{proof}
Apply Proposition \ref{P:genStone} with $T\coloneqq S$. 
\end{proof}
Clearly, if $S$ is a symmetric operator then $S\cap S^*=S$. Hence we retrieve   \cite{Stone}*{Theorem 2.9} by M. H. Stone as an immediate consequence (cf. also \cite{SZTZS2019}*{Corollary 6.7}): 
\begin{corollary}
Every surjective symmetric operator is densely defined and self-adjoint. 
\end{corollary}
In the next result we give a necessary and sufficient condition for an operator $S$ to be identical with the adjoint of a given operator $T$. 
\begin{theorem}
Let $\hil,\kil$ be real or complex Hilbert spaces and let $S:\hil\to\kil$ and $T:\kil\to\hil$ be (not necessarily densely defined or closed) linear operators. The following two statements are equivalent:
\begin{enumerate}[label=\textup{(\roman*)}]
    \item $T$ is densely defined and $S=T^*$,
    \item \begin{enumerate}[label=\textup{(\alph*)}]
        \item $(\ran T)^{\perp}=\ker S$,
        \item $\ran S+\ran T^*\subset \ran (S\cap T^*)$.
    \end{enumerate}
\end{enumerate}
\end{theorem}
\begin{proof}
It is obvious that (i) implies (ii). Assume now (ii) and for sake of brevity introduce the operator
\begin{equation*}
    S^{}_0\coloneqq S\cap T^*.
\end{equation*}
We start by establishing that $T$ is densely defined. Let $g\in(\dom T^*)^{\perp}$, then $(0,g)\in T^*$, i.e., $g\in \ran T^*$. By (ii) (a),
\begin{equation*}
    T^*g=S_0h=Sh
\end{equation*}
for some $h\in\dom S_0$. Then it follows  that $(h,Sh)\in T^*$ and therefore 
\begin{equation*}
    \sip{Tk}{h}=\sip{k}{Sh}=\sip{k}{g}=0,\qquad k\in\dom T,
\end{equation*}
whence we infer that $h\in(\ran T)^{\perp}$. Again by (ii) (a) we have $h\in\ker S$ and thus $g=Sh=0$. This proves that $T$ is densely defined and as a consequence, $T^*$ is an operator. Next we  prove that 
\begin{equation}\label{E:3.2-TsubS}
    T^*\subset S.
\end{equation}
To see this consider $g\in\dom T^*$. By (ii) (b),
\begin{equation*}
    T^*g=S_0h=Sh=T^*h
\end{equation*}
for some $h\in\dom S_0$. Then it follows that $T^*(g-h)=0$, i.e., 
\begin{equation*}
    g-h\in\ker T^*=(\ran T)^\perp=\ker S.
\end{equation*}
Consequently, $g=(g-h)+h\in\dom S$ and $Sg=Sh=T^*g$, which proves \eqref{E:3.2-TsubS}. It only remains to show that the converse inclusion
\begin{equation}\label{E:3.2-SsubT}
    S\subset T^*
\end{equation}
holds also true. For let $g\in \dom S$ and choose $h\in\dom S_0$ such that 
\begin{equation*}
    Sg=S_0h=T^*h=Sh.
\end{equation*}
Then $g-h\in\ker S=(\ran T)^\perp=\ker T^*$ whence we get $g=(g-h)+h\in\dom T^*$ and $T^*g=T^*h=Sg$, which proves \eqref{E:3.2-SsubT}.
\end{proof}

A celebrated theorem by J. von Neumann \cite{vonNeumann} states that $S^*S$ and $SS^*$ are positive and selfadjoint operators provided that $S$ is a densely defined and closed operator between $\hil$ and $\kil.$ In that case, $I+S^*S$ and $I+SS^*$ are both surjective. In \cite{SZ-TZS:reversed} it has been proved that the converse is also true: If $I+S^*S$ and $I+SS^*$ are both surjective operators then $S$ is necessarily closed (cf. also \cite{Gesztesy-Schmudgen}).  Below, as the main result of the paper, we establish an improvement of Neumann's theorem:
\begin{theorem}\label{T:RanI+ST}
Let $\hil,\kil$ be real or complex Hilbert spaces and let $S:\hil\to\kil$ and $T:\kil\to\hil$ be linear operators and introduce the operators $S_0\coloneqq S\cap T^*$ and $T_0\coloneqq T\cap S^*$. The following statements are equivalent:
\begin{enumerate}[label=\textup{(\roman*)}]
    \item $S,T$ are both densely defined and they are adjoint of each other: $S^*=T$ and $T^*=S$,
    \item $\ran (I+T_0S_0)=\hil$ and $\ran (I+S_0T_0)=\kil$. 
\end{enumerate}
\end{theorem}
\begin{proof}
It is clear that (i) implies (ii). To prove the converse implication observe first that
\begin{equation*}
    \sip{S^{}_0u}{v}=\sip{u}{T^{}_0v},\qquad u\in\dom S^{}_0, v\in \dom T^{}_0.
\end{equation*}
We start by showing that $S^{}_0$ is densely defined. Take a vector $g\in(\dom S_0)^{\perp}$, then there is $u\in\dom S^{}_0$ such that $g=u+T^{}_0S^{}_0u$. Consequently, 
\begin{equation*}
    0=\sip{u}{g}=\sip{u}{u}+\sip{T^{}_0S^{}_0u}{u}=\|u\|^2+\|S^{}_0u\|,
\end{equation*}
whence $u=0$, and therefore also $g=0$. It is proved analogously that  $T^{}_0$ is densely defined too, and therefore the adjoint relations $S_0^{*}$ and $T_0^{*}$ are operators such that $S^{}_0\subset T_0^*$ and $T^{}_0\subset S_0^*$.

We are going to prove now that $S^{}_0$ and $T^{}_0$ are adjoint of each other, i.e, 
\begin{equation}\label{E:T.3.3-adjofeachother}
    S_0^*=T^{}_0,\qquad T_0^*=S_0^{}.
\end{equation}
Consider a vector $g\in\dom T_0^*$  and take $u\in \dom S^{}_0$ and $v\in \dom T^{}_0$ such that 
\begin{equation*}
    g=u+T_0S_0u\qquad\mbox{and}\qquad T^{*}_0g=v+S^{}_0T^{}_0v.
\end{equation*}
Since $u$ is in $\dom T_0^*$ we infer that $T_0S_0u\in\dom T_0^*$ and hence 
\begin{equation*}
    T_0^*g=T_0^*u+T_0^*T^{}_0S^{}_0u.
\end{equation*}
It follows then that 
\begin{align*}
    0&=v-T_0^*u+S^{}_0T_0^{}v-T^{*}_0T^{}_0S^{}_0u=(I+T^{*}_0T^{}_0)(v-S^{}_0u)
\end{align*}
which yields $v=S_0u\in\dom T_0.$ As a consequence we obtain that
\begin{equation*}
    g=u+T^{}_0S^{}_0u=v+T^{}_0v,
\end{equation*}
and therefore that $g\in\dom S_0$. This proves the first equality of \eqref{E:T.3.3-adjofeachother}. The second one is proved in a similar way. 

Now we  can complete the proof easily: since $S_0\subset T^*$ and $T_0\subseteq T$ it follows that 
\begin{equation*}
    T_0^*=S^{}_0\subset T^*\subset T_0^*, 
\end{equation*}
whence $T^*=T_0^*=S^{}_0$, and therefore $T^*\subset S$.   On the other hand, $T_0\subset S^*$ implies
\begin{equation*}
    S\subset S^{**}\subset T^{*}_0=T^*,
\end{equation*}
whence we conclude that  $S=T^*$. It can be proved in a similar way that  $T=S^*$.
\end{proof}
As an immediate consequence we conclude the following result:
\begin{corollary}
Let $\hil$ and $\kil$ be real or complex Hilbert spaces and let $S:\hil\to\kil$ be a densely defined operator. The following statements are equivalent:
\begin{enumerate}[label=\upshape{(\roman*)}]
    \item $S$ is closed,
    \item $S^*S$ and $SS^*$ are self-adjoint operators,
    \item $\ran (I+S^*S)=\hil$ and $\ran (I+SS^*)=\kil$.
\end{enumerate}
\end{corollary}
\begin{proof}
Apply Theorem \ref{T:RanI+ST} with $T\coloneqq S^*$.
\end{proof}
%%%%%%%%%%%%%%%%%%%%%%%%%%%%%%%%%%%%%%%%%%%%%%%%%%%%%
In the ensuing theorem we provide a renge-kernel characterization of operators $T$ that are identical with the adjoint $S^*$ of a densely defined symmetric operator $S$. We stress that  no condition on the closedness of the operator or density of the domain is imposed. On the contrary: we get those properties from the other conditions. 

%We stress that no a priori condition on the closedness  or on the density of the domain is assumed. On the contrary: these follow automatically from the other assumptions.
\begin{theorem}
Let $\hil$ be a real or complex Hilbert space and let $T:\hil\to\hil$ be a (not necessarily densely defined or closed) linear operator and let $T_0\coloneqq T\cap T^*$. The following two statements are equivalent:
\begin{enumerate}[label=\textup{(\roman*)}]
    \item there exists a densely defined symmetric operator $S$ such that $S^*=T$,
    \item \begin{enumerate}[label=\textup{(\alph*)}]
        \item $\ker T=(\ran T^*)^{\perp}$,
        \item $\ran T_0=\ran T^{**}=\ran T^*$.
    \end{enumerate}
\end{enumerate}
In particular, if any of the equivalent conditions (i), (ii) is satisfied then $T$ is a densely defined and closed operator such that  $T^*\subset T$. 
\end{theorem}
\begin{proof}
It is straightforward that (i) implies (ii) so we only prove the converse. We start by proving that $T$ is densely defined. Take $g\in(\dom T)^{\perp}$, then $g\in\mul T^*\subseteq \ran T^*$. By (ii) (b), there exists $h\in\dom T_0$ such that $g=T_0h=Th$. Consequently, $(h,g)\in T^*$ and for every $f\in\dom T$,
\begin{equation*}
    \sip{Tf}{h}=\sip{f}{g}=0,
\end{equation*}
which yields  $h\in(\ran T)^{\perp}$. Observe that (ii) (a) and (b) together imply that 
\begin{equation}\label{E:ranTperp=kerT}
    (\ran T)^\perp=\ker T,
\end{equation}
whence we infer that $h\in \ker T$ and therefore that $g=Th=0$. This means that $T^*$ is a (single valued) operator. Or next claim is to show that 
\begin{equation}\label{E:T*subsetT}
    T^*\subset T.
\end{equation}
To this end, let $g\in \dom T^*$, then $T^*g=T_0h$ for some $h_0\in\dom T_0$. From inclusion $T_0\subset T^*$ we conclude that 
$g-h\in\ker T^*=(\ran T)^\perp$, thus $g=(g-h)+h\in \dom T$ and 
\begin{equation*}
    Tg=Th=T_0h=T^*g,
\end{equation*}
which proves \eqref{E:T*subsetT}. Next we show that $T^*$ is densely defined too, i.e., $T$ is closable. To this end condider a vector $g\in(\dom T^*)^\perp=\mul T^{**}$. Since $\mul T^{**}\subseteq \ran T^{**}$,  we can find  a vector $h\in\dom T_0$ such that $g=T_0h$. For every $k\in\dom T^*$,
\begin{equation*}
    \sip{h}{T^*k}=\sip{Th}{k}=\sip{g}{k}=0,
\end{equation*}
thus $h\in (\ran T^*)^{\perp}$. By (ii) (a) we infer that $h\in\ker T$ and hence $g=Th=0$, hence $(\dom T^*)^\perp=\{0\}$, as it is claimed. Finally we show that $T$ is closed. Take $g\in\dom T^{**}$, then $T^{**}g=Th$ for some $h\in\dom T$, according to assumption  (ii) (b). Hence $g-h\in\ker T^{**}=(\ran T^*)^{\perp}$, thus $g-h\in \ker T$ because of (ii) (a). Consequently, $g=(g-h)+h\in\dom T$ which proves identity $T=T^{**}$. Summing up, $S\coloneqq T^*$ is a densely defined operator such that $S\subset T=S^*$. In other words, $T$ is identical with the adjoint $S^*$ of the symmetric operator $S$.    
\end{proof}
%%%%%%%%%%%%%%%%%%%%%%%%%%%%%%%%%%%%%%%%%%%%%%%%%%%%%%%%%%%%%%%%%
\section{Characterizations involving resolvent norm estimations}

Let $\hil$ and $\kil$ be real or complex Hilbert spaces. For given two linear operators $S:\hil\to\kil$ and $T:\kil\to\hil$, let us consider the operator matrix
\begin{equation*}
    M_{S,T}\coloneqq \opmatrix{0}{-T}{S}{0}, 
\end{equation*}
acting on the product Hilbert space $\hil\times \kil$. More precisely, $ M_{S,T}$ is an operator acting on its domain $\dom  M_{S,T}(\lambda)\coloneqq\dom S\times \dom T$ by 
\begin{equation*}
    M_{S,T}(h,k)\coloneqq (-Tk,Sh) \qquad (h\in\dom S, k\in\dom T).
\end{equation*}
Assume that a real or complex number $\lambda\in\dupK$ belongs to the resolvent set $\rho(M_{S,T})$, which means that 
\begin{equation*}
    M_{S,T}-\lambda =\opmatrix{-\lambda}{-T}{S}{-\lambda}
\end{equation*}
has an everywhere defined bounded inverse. In that case, for brevity's sake, we  introduce the notation
\begin{equation*}
    R_{S,T}(\lambda)\coloneqq (M_{S,T}-\lambda )^{-1}
\end{equation*} 
for the corresponding resolvent operator. 

In the present section we are going to establish some criteria, by means of norms of the resolvent operator $R_{S,T}(\lambda)$, under which the operators $S$ and $T$ are adjoint of each other. Our approach is motivated by the classical paper of T. Nieminen \cite{Nieminen} (cf. also \cite{Sandovici}). We emphasize that our  framework is more general than that of \cite{Nieminen} for many ways: we do not assume that the operators under consideration are densely defined or closed, and also the underlying space may be real or complex.
\begin{theorem}\label{T:Nieminen}
Let $S:\hil\to\kil$ and $T:\kil\to\hil$ be linear operators between the real or complex Hilbert spaces $\hil$ and $\kil$. The following assertions are equivalent:
\begin{enumerate}[label=\textup{(\roman*)}, labelindent=\parindent]
    \item $S$ and $T$ are densely defined such that  $S^*=T$ and $T^*=S$,
    \item every non-zero real number $t$ belongs to the resolvent set of $M_{S,T}$ and  
    \begin{equation}\label{E:Nieminen}
        \|R_{S,T}(t) \|\leq \frac{1}{\abs t},\qquad \forall t\in\dupR, t\neq 0.
    \end{equation}
\end{enumerate}
\end{theorem}
\begin{proof}
Let us start by proving that (i) implies (ii). Assume therefore that $S$ is densely defined and closed and that $T=S^*$. Consider a non-zero real number $t$ and a pair of vectors $h\in\dom S$ and $k\in\dom S^*$, then we have
\begin{align*}
    \left\|\opmatrix{t}{-S^*}{S}{t}\pair hk\right\|^2&=\|th-S^*k\|^2+\|Sh+tk\|^2\\
                &=t^2[\|h\|^2+\|k\|^2]+\|Sh\|^2+\|S^*k\|^2\\
                &\geq  t^2\left\|\pair hk\right\|^2
\end{align*}
which implies that $M_{S,T}+t$ is bounded from below and the norm of its inverse $R_{S,T}(-t)$ satisfies \eqref{E:Nieminen}. However it is not yet clear that $R_{S,T}(-t)$ is everywhere defined. But observe that 
\begin{equation*}
    \opmatrix{t}{-S^*}{S}{t}\pair hk=t\bigg(\pair{h}{\tfrac1tSh} +\pair{-\tfrac1tS^*k}{k}\bigg)
\end{equation*}
whence we get 
\begin{equation}\label{E:ranMST}
    \ran(M_{S,T}+t)=\tfrac1tS\oplus W (\tfrac1tS^*),
\end{equation}
where $W$ is the `flip' operator $W(k,h)\coloneqq (-h,k)$.
Since $S$ is densely defined and closed according to our hypotheses, the  subspace on the right hand side of \eqref{E:ranMST} is equal to $\hil\times \kil$. This proves statement (ii).
 
For the converse direction, observe that \eqref{E:Nieminen} implies 
\begin{equation*}
    \left\|\opmatrix{t}{-T}{S}{t}\pair hk\right\|^2\geq t^2\left\|\pair hk\right\|^2,\qquad h\in\dom S, k\in\dom T.
\end{equation*}
Hence from \eqref{E:Nieminen} we conclude that 
\begin{align*}
    0&\geq\|Sx\|^2+\|Ty\|^2+t\{\sip{Sx}{y}-\sip{x}{Ty}+\sip{y}{Sx}-\sip{Ty}{x}\}\\
     & =\|Sx\|^2+\|Ty\|^2+2t  \operatorname{Re}\{\sip{Sx}{y}-\sip{x}{Ty}\}
\end{align*}
for every $t\in\dupR$. Consequently, 
\begin{equation*}
    \operatorname{Re} \sip{Sx}{y}=\operatorname{Re} \sip{x}{Ty},\qquad x\in\dom S, y\in\dom T.
\end{equation*}
In the real Hilbert space case it is straightforward that  $S$ and $T$ are adjoint to each other. In the complex case, replace $x$ by $ix$ to get  
\begin{equation*}
    \operatorname{Im} \sip{Sx}{y}=\operatorname{Im} \sip{x}{Ty},\qquad x\in\dom S, y\in\dom T.
\end{equation*}
So, in both real and complex cases, we obtained that  $S\subset T^* $ and $T\subset S^*$. With  notation of Theorem \ref{T:RanI+ST} this means that  $S_0=S$ and $T_0=T$. Since we have \begin{equation*}
    \opmatrix{I}{-T}{S}{I}= M^{}_{S,T}+1 ,\qquad \opmatrix{I}{T}{-S}{I}=- [M^{}_{S,T}-1],
\end{equation*}
we conclude that 
\begin{equation*}
    \opmatrix{I+TS}{0}{0}{I+ST}= -M^{}_{S,T}(1) M^{}_{S,T}(-1)
\end{equation*}
is a surjective operator onto $\hil\times \kil$, which entails  $\ran (I+TS)=\hil$ and $\ran (I+ST)= \kil$. An immediate application of Theorem \ref{T:RanI+ST} completes the proof. 
\end{proof}
As an immediate consequence of Theorem \ref{T:Nieminen} we can establish the following characterizations of self-adjoint, skew-adjoint and unitary operators. %(Observe that neither symmetry of the operator nor density of the assumption on  the domain is imposed.
\begin{corollary}
Let $\hil$ be a real or complex Hilbert space. For a linear operator $S:\hil\to\hil$  the following assertions are equivalent:
\begin{enumerate}[label=\textup{(\roman*)}, labelindent=\parindent]
    \item $S$ is  densely defined and self-adjoint,
    \item Every non-zero real number $t$ is in the resolvent set of $M_{S,-S}$ and  
    \begin{equation}%\label{E:Nieminen}
        \|R_{S,-S}(t) \|\leq \frac{1}{\abs t},\qquad \forall t\in\dupR, t\neq 0.
    \end{equation}
\end{enumerate}
\end{corollary}
\begin{proof}
Apply Theorem \ref{T:Nieminen} with $T\coloneqq S$ to conclude the the desired equivalence.
\end{proof}
%%%%%%%%%%%%%%%%
\begin{corollary}
Let  $\hil$ be a real or complex Hilbert space. For a linear operator $S:\hil\to\hil$  the following assertions are equivalent:
\begin{enumerate}[label=\textup{(\roman*)}, labelindent=\parindent]
    \item $S$ is  densely defined and skew-adjoint,
    \item every non-zero real number $t$ is in the resolvent set of $M_{S,S}$ and  
    \begin{equation} 
        \|R_{S,S}(t) \|\leq \frac{1}{\abs t},\qquad \forall t\in\dupR, t\neq 0.
    \end{equation}
\end{enumerate}
\end{corollary}
\begin{proof}
Apply Theorem \ref{T:Nieminen} with $T\coloneqq -S$.
\end{proof}
\begin{corollary}
Let $\hil$ and $\kil$ be a real or complex Hilbert spaces. For a linear operator $U:\hil\to\kil$  the following assertions are equivalent:
\begin{enumerate}[label=\textup{(\roman*)}, labelindent=\parindent]
    \item $U$ is a unitary operator,
    \item $\ker U=\{0\}$, every non-zero real number $t$ is in the resolvent set of $M_{U,U^{-1}}$ and  
    \begin{equation} 
        \|R_{U,U^{-1}}(t) \|\leq \frac{1}{\abs t},\qquad \forall t\in\dupR, t\neq 0.
    \end{equation}
\end{enumerate}
\end{corollary}
\begin{proof}
An application of Theorem \ref{T:Nieminen} with $S\coloneqq U$ and $T\coloneqq U^{-1}$ shows that $U$ is densely defined and closed such that $U^*=U^{-1}$. Hence, $\ran U^*\subseteq \dom U$. Since we have $\ran U^*+\dom U=\hil$ for every densely defined closed operator $U$, we infer that $\dom U=\hil$ and therefore $U$ is a unitary operator.
\end{proof}

%%%%%%%%%%%%%%%%%%%%%%%%%%%%%%%%%%
%%%%%%%%%%%%%%%%%%%%%%%%%%%%%%%%%%%%%%%%%%%%%%%%%%%%%%%%
\bibliographystyle{abbrv}

\begin{thebibliography}{10}

\bibitem{Arens}
R. Arens, Operational calculus of linear relations,  \emph{Pacific J. Math.} \textbf{11} (1961), 9--23.

\bibitem{BehrndtHassideSnoo}
 J. Behrndt, S. Hassi, H.S.V. de Snoo, \emph{Boundary Value Problems, Weyl Functions, and Differential Operators}
  Monographs in Mathematics, vol 108. Birkhäuser, Cham, 2020.

 
\bibitem{Gesztesy-Schmudgen}
F. Gesztesy, K. Schm\"udgen, On a theorem of Z. Sebesty\'en and Zs. Tarcsay, \emph{Acta Sci. Math. (Szeged)}, 85 (2019), 291--293.
 
 \bibitem{Hassi2009a}
S.~Hassi, H.S.V.~de~Snoo, and F. H. Szafraniec,
\newblock Componentwise and canonical decompositions of linear relations,
\newblock \emph{Dissertationes Mathematicae}, \textbf{465} (2009), 59pp.

 
\bibitem{Nieminen}
T. Nieminen, A condition for the self-adjointness of a linear operator,
\emph{Ann. Acad. Sci. Fenn. Ser.} A I No 316. 

\bibitem{vonNeumann1930}
J.~v.~Neumann,
\newblock Allgemeine Eigenwerttheorie hermitescher Funktionaloperatoren,
\newblock {\em Mathematische Annalen}, \textbf{102} (1930), 49--131.

\bibitem{vonNeumann}
J.~v. Neumann,
\newblock \"Uber adjungierte Funktionaloperatoren,
\newblock {\em Annals of Mathematics}, \textbf{33} (1932), 294--310.
 
\bibitem{Popovici}
D. Popovici and Z. Sebesty\'en,
\newblock On operators which are adjoint to each other,
\newblock {\em Acta Sci. Math. (Szeged)}, \textbf{80} (2014), 175--194.
 
\bibitem{Sandovici}
 M. Roman and A. Sandovici, A note on a paper by Nieminen, \emph{Results Math.} \textbf{74} (2019), no. 2, Art. 73, 6 pp.
  

\bibitem{Schmudgen}
K. Schm\"udgen, \emph{Unbounded self-adjoint operators on Hilbert space}, Vol. 265. Springer Science \& Business Media, 2012.



\bibitem{Characterization}
Z. Sebesty\'en and Zs. Tarcsay,
\newblock Characterizations of selfadjoint operators, \emph{Studia Sci. Math. Hungar.} \textbf{50} (2013), 423--435.

\bibitem{SZ-TZS:reversed}
Z. Sebesty\'en and Zs. Tarcsay, A reversed von Neumann theorem, \emph{Acta Sci. Math. (Szeged)}, \textbf{80} (2014), 659--664.


\bibitem{Charess}
Z. Sebesty\'en and Zs. Tarcsay,
\newblock Characterizations of essentially selfadjoint  and skew-adjoint operators, \emph{Studia Sci. Math. Hungar.},  \textbf{52} (2015), 371--385.


\bibitem{SZTZS2015}
Z. Sebesty\'en and Zs. Tarcsay, Adjoint of sums and products of operators in Hilbert spaces, \emph{Acta Sci. Math. (Szeged)}, \textbf{82} (2016), 175--191.

\bibitem{Squareroot}
Z. Sebesty\'en and Zs. Tarcsay, On the square root of a positive selfadjoint operator, \emph{Periodica Mathematica Hungarica},  \textbf{75} (2017) 268--272. 

\bibitem{SZTZS2019}
Z. Sebesty\'en and Zs. Tarcsay, On the adjoint of Hilbert space operators, \emph{Linear and Multilinear Algebra}, \textbf{67} (2019),  625--645.

%\bibitem{SZTZS_squareroot}
%Z. Sebesty\'en and Zs. Tarcsay, On square root of positive selfadjoint operators, \emph{Period. Math. Hungar.}, \textbf{75}  (2017), 268-272.


\bibitem{Stone}
M. H. Stone, \emph{Linear Transformations in Hilbert Spaces and their Applications to
Analysis}, Amer. Math. Soc. Colloq. Publ. \textbf{15}, Amer. Math. Soc., 1932.
 
\end{thebibliography}

\end{document}